\documentclass[11pt]{article}
\usepackage{amsmath,amsfonts,amsthm,amssymb,ulem}
\usepackage{times}

\theoremstyle{plain}
\newtheorem{thm}{Theorem}
\newtheorem*{thm*}{Theorem}
\newtheorem*{Theorem A}{Theorem A}
\newtheorem*{Theorem B}{Theorem B}
\newtheorem*{Theorem C}{Theorem C}
\newtheorem{lem}[thm]{Lemma}

\newtheorem{prop}[thm]{Proposition}
\theoremstyle{definition}
\newtheorem{defn}[thm]{Definition}
\theoremstyle{remark}

\newcommand{\pl}{\partial}
\newcommand{\na}{\nabla}

\newcommand{\lt}{\left}
\newcommand{\rt}{\right}
\newcommand{\rw}{\rightarrow}

\newcommand{\R}{\mathbb{R}}

\newcommand{\mf}{\mathbf}

\newcommand{\mr}{\mathbf{r}}
\newcommand{\my}{\mathbf{y}}
\newcommand{\mz}{\mathbf{z}}
\newcommand{\mw}{\mathbf{w}}
\newcommand{\bY}{\bar{Y}}

\usepackage{tensor}
\usepackage{appendix}
\usepackage{color}

\begin{document}
\fontsize{12}{16pt plus.4pt minus.3pt}\selectfont
%\title{}\author{}\date{}\maketitle
%\begin{abstract}\end{abstract}
%\begin{CJK}{UTF8}{cwmb}\end{CJK}
%%
%%%cwmb明體,cwkb楷書,cwfsb仿宋體,cwhbb粗黑體,cwyb圓體
%%

\title{On Isometric Embeddings into Anti-de Sitter Space-times}

\author{Ye-Kai Wang \\
Department of Mathematics\\
Columbia University\\
yw2293@math.columbia.edu\\
\and
Chen-Yun Lin\\
Department of Mathematics\\
National Taiwan University\\
chenyunlin@ntu.edu.tw
}
\maketitle

\begin{abstract}
We show that any metric on $S^2$ with Gauss curvature $K \geq -\kappa$ admits a $C^{1,1}$-isometric embedding into the hyperbolic space with sectional curvature $-\kappa$. We also give a sufficient condition for a metric on $S^2$ to be isometrically embedded into anti-de Sitter spacetime with the prescribed cosmological time function.
 
\end{abstract}

\section{Introduction}
Weyl's isometric embedding theorem states that
\begin{thm*}
Let $\bar{\sigma}$ be a smooth metric on $S^2$ with positive Gauss curvature. Then there exists a smooth isometric embedding $i:(S^2,\bar{\sigma}) \rw \mathbb{R}^3$ which is unique up to congruence. 
\end{thm*}
Weyl's theorem was proved independently by Nirenberg \cite{Nirenberg53: Weyl problem} and Alexandrov-Pogorelov \cite{Pogorelov52} using different approaches. Nirenberg used the continuity method that is more familiar to geometric analysts. The Alexandrov-Pogorelov approach consists of two steps. First Alexandrov exhibited a generalized solution as the limit of polyhedra and then Pogorelov proved the regularity of this generalized solution. 

Later, Pogorelov generalized the Weyl theorem to hyperbolic space $\mathbb{H}^3_{-\kappa}$ with sectional curvature $-\kappa$:  
\begin{thm*}\cite[Theorem 2]{Pogorelov64}
Let $\bar{\sigma}$ be a smooth metric on $S^2$ with Gauss curvature $K > -\kappa.$ Then there exists a smooth isometric embedding $i:(S^2, \bar{\sigma}) \rw \mathbb{H}^3_{-\kappa}$.
\end{thm*}

Another question concerning Weyl's theorem is what happens if we assume the Gauss curvature is merely nonnegative instead of positive? Guan-Li \cite{Guan-Li94} and Hong-Zuily \cite{Hong-Zuily95} independently proved the following (see also \cite{Iaia92} for prior results)
\begin{thm*}
Let $\bar{\sigma}$ be a $C^4$ Riemannian metric on $S^2$ with Gauss curvature $K \geq 0.$ Then there exists a $C^{1,1}$ isometric embedding $i: (S^2,\bar{\sigma}) \rw \mathbb{R}^3.$
\end{thm*}

In this paper, we follow Nirenberg's approach to solve the isometric embedding problem into hyperbolic space by the continuity method. Our main result is
\begin{Theorem A}
Let $\bar{\sigma}$ be a smooth metric on $S^2$ with Gauss curvature $K \geq-\kappa$. Then there exists a $C^{1,1}$ isometric embedding $i: (S^2, \bar{\sigma}) \rw \mathbb{H}^3_{-\kappa}$. 
\end{Theorem A}

Given an embedding $i:S^2 \rw \mathbb{H}^3_{-\kappa}$ and a trivialization $\varphi: \mathbb{H}^3_{-\kappa} \rw \mathbb{R}^3$, let $\mr= \varphi \circ i$ be the ``position vector". After choosing local coordinates $\{u^a\}, a=1,2$ on $S^2$, the isometric embedding equation can be written as the nonlinear first order partial differential system 
\begin{align}
g_{ij} (\mr(x)) \frac{\pl \mr^i}{\pl u^a}(x) \frac{\pl \mr^j}{\pl u^b} = \bar{\sigma}_{ab}(x).
\end{align}
Unlike the case of Euclidean space, the form of the equation depends on the choice of the trivialization. For most of this paper, we choose to use the static coordinate chart of $\mathbb{H}^3_{-\kappa}$ in which the manifold is identified with $\mathbb{R}^3$ and the metric has the form
\begin{align*}
g = \frac{1}{f^2} dr^2 + r^2 g_{S^2},
\end{align*}
where $f= \sqrt{1+\kappa r^2}$ is called the static potential and $g_{S^2}$ denotes the standard metric on $S^2.$ 

%%part of proof--move below 
%
%% 

We now outline the proof of Theorem A. First of all, the normalized Ricci flow \cite{Hamilton88, Chow91} provides an one-parameter family of smooth metrics $\sigma_t, t \in [0,\infty)$ on $S^2$ such that $\sigma_0 = \bar{\sigma}$ and $\sigma_t$ converges to a metric $\sigma_\infty$ with constant Gauss curvature. Moreover, the Gauss curvature of $\sigma(t)$ is greater than $-\kappa$ for $t>0.$ 

Let $I \subset [0,\infty)$ be the set of parameters such that $\sigma_t $ can be isometrically embedded into  $\mathbb{H}^3_{-\kappa}$ as a closed convex $C^{1,1}$ surface. The goal is to show that $I$ is nonempty, open, and closed. As a result, $I = [0,\infty)$. 

We start with the openness part which states that if a metric $\sigma$ can be isometrically embedded into $\mathbb{H}^3_{-\kappa}$, so can any small perturbation of $\sigma$. The first step is to understand the infinitesimal deformation. We show that all infinitesimal deformations come from the isometry of hyperbolic space. The next step is to solve the linearized equation. Remarkably, Nirenberg reduced the system of equations to a single scalar equation and used Hilbert theory to solve it. We modify Nirenberg's argument so that it fits the geometry of hyperbolic space. At last, we utilize the contraction mapping principle to solve the nonlinear equation.

Nonemptiness follows from openness. Since the limit metric $\sigma_\infty$ can be isometrically embedded into $\mathbb{H}^3_{-\kappa}$ (with the same image as a round sphere), $[T_0,\infty) \subset I$ for some large $T_0$. 

The closedness part boils down to an a priori estimate of the mean curvature $H$ of convex surfaces in $\mathbb{H}^3_{-\kappa}$. Weyl's original approach is to apply the maximum principle to the mean curvature of the surface, which fails when the metric has negative Gauss curvature somewhere. We manage to find a test function to overcome the difficulty. We show that
\begin{Theorem B}
Let $\Sigma$ be a closed convex surface in $\mathbb{H}^3_{-\kappa}$, normalized so that $\Sigma$ is centered at the origin. Then
\begin{align*}
\max_\Sigma H \leq C,
\end{align*}
for some constant $C$ depending only on $\|f\|_{C^0(\Sigma)}$, and $\| K\|_{C^2(\Sigma)}.$ Here, $f=\sqrt{1+\kappa r^2}$ is the static potential and $K$ denotes the Gauss curvature of $\Sigma$.
\end{Theorem B}
%explain that $f$ is the static potential
During the preparation of this paper, we learned that Guan and Lu \cite{Guan-Lu13} independently obtained the same estimate as in Theorem B.
Note that for convex surfaces in hyperbolic space with sectional curvature $-1$, Chang and Xiao \cite{Chang-Xiao12} proved an a priori bound on the mean curvature with an argument based on Pogorelov's estimate using the additional assumption that the set $\{K= -1 \}$ consists of only finitely many points. Our new argument does not require this finiteness condition.

With Nirenberg's estimates  for uniformly elliptic equations in two dimensions, we show that if $\bar{\sigma}$ has $K > -\kappa,$ the isometric embedding constructed by the continuity method is actually smooth. Hence we recover Pogorelov's result. 

The Weyl theorem plays a prominent role in the study of quasilocal mass in general relativity. Let $\Sigma$ be a 2-surface in an initial data set. Suppose the induced metric $\sigma$ has positive Gauss curvature. By Weyl's theorem, there exists a unique (up to congruence) isometric embedding $i: (\Sigma, \sigma) \rw \mathbb{R}^3$. Let $H_0$ and $H$ be the mean curvature of $i(\Sigma) \subset \R^3$ and $\Sigma$ in the initial data set, respectively. The Brown-York mass \cite{Brown-York92, Brown-York93} is defined as
\begin{align*}
m_{BY}(\Sigma) = \frac{1}{8\pi} \left( \int_\Sigma H_0 d\mu - \int_\Sigma H d\mu \right) .
\end{align*}
In \cite{Wang-Yau08}, Wang and Yau proposed a new quasilocal mass for spacelike 2-surfaces in spacetime. The definition requires the existence of an isometric embedding of the surface into Minkowski spacetime. It is also interesting to consider other backgrounds. Theorem C provides a necessary condition for the existence of isometric embeddings into the cosmological chart of anti-de Sitter spacetime with prescribed cosmological time function.

\begin{Theorem C}
Given a smooth metric $\sigma$ and a smooth function $s$ on $S^2.$ Let $\na$ and $\Delta$ denote the gradient and Laplace operator with respect to $\sigma$, respectively. Suppose the Gauss curvature $K$ and the function $s$ satisfy 
\begin{align}
K +  \frac{S'}{S} \Delta s - \lt( \frac{S''}{S} - \frac{S'^2}	{S^2} \rt) |\nabla s|^2  + 
 ( 1+|\nabla s|^2 )^{-1} \lt(\frac{ \det(\nabla^2 s) }{\det \sigma} - \frac{S'}{S} \nabla^a s \nabla^b s\nabla_a \nabla_b s\rt)>-\kappa \label{Gauss curvature of the projection}
\end{align}
where $S(t) = \cos (\sqrt{\kappa} t)$ is the scale factor of anti-de Sitter spacetime in its cosmological chart. Then there exists a unique spacelike isometric embedding $i :(S^2,\sigma) \rw AdS$ with prescribed cosmological time function $s$.
\end{Theorem C}
 
The rest of this paper is organized as follows. In section 2, we describe the Killing and conformal Killing vector fields on hyperbolic space and list the necessary formulae. In section 3, we study the infinitesimal rigidity of convex surfaces in $\mathbb{H}^3_{-\kappa}$. Section 4 is the most important of the paper. In this section, we construct a path of metrics and prove openness on this path. Then we establish the a priori estimate which proves Theorem B. Together these results provide a proof of Theorem A. Finally, in the last section we discuss isometric embeddings into the anti-de Sitter spacetime and the proof of Theorem C.

{\it Acknowledgment}. We are indebted to Professor Mu-Tao Wang for his encouragement, insightful comments and assistance through our work. We would also like to thank Po-Ning Chen, Xiangwen Zhang and Thomas Nyberg for informative discussions, and professors Pengfei Guan and Joel Spruck for their interest in our work.

\section{The Killing vectors on hyperbolic space}
In the hyperboloid model, $\mathbb{H}^n_{-\kappa}$ is identified as a hypersurface in $
\mathbb{R}^{n,1}$ given by $-(x^0)^2 + (x^1)^2 + \cdots + (x^n)^2 = -\frac{1}{\kappa}.$ The projection of the hyperboloid onto the hyperplane $\{x_0=0\}$ provides the static coordinates $(\mathbb{R}^n,g)$ of $\mathbb{H}^n_{-\kappa}.$ In polar coordinates, the induced metric is expressed as
\begin{align*}
g= \frac{1}{1+\kappa r^2}dr^2 + r^2 g_{S^{n-1}}. 
\end{align*}
The hyperbolic space admits a conformal Killing vector field 
\begin{equation}\label{X}
X = r \sqrt{1+\kappa r^2} \frac{\pl}{\pl r} = rf \frac{\partial}{\partial r},
\end{equation}
where as earlier the function $f=\sqrt{1+\kappa r^2}$  
is the static potential. In the hyperboloid model, the static potential can be rewritten as $f=\sqrt{\kappa} x^0$ and we have that
$$
X = \sqrt{\kappa}\left( \left( (x^0)^2 - \frac{1}{\kappa} \right) \frac{\pl}{\pl x^0} + x^0 \sum_{i=1}^3 x^i \frac{\pl}{\pl x^i}\right).
$$

Let $D$ denote the covariant derivative on the hyperbolic space.
\begin{prop} The gradient of the static potential $f=\sqrt{1+\kappa r^2}$ and the covariant derivative of the conformal Killing vector field $X=rf\frac{\partial}{\partial r}$ on $\mathbb{H}^n_{-\kappa}$ are the following:
\begin{align}
Df =\kappa X \label{grad f};\\
D_\xi X =f\xi \label{conformal Killing}.
\end{align}
\end{prop}
\begin{proof}
We calculate in polar coordinates. Let $\theta^i, i=1, \ldots, n-1$ be any local coordinates on $S^{n-1}.$ We write $\xi = \xi^r \frac{\pl}{\pl r} + \xi^i \frac{\pl}{\pl \theta^i}$.
\begin{align*}
Df &= g^{rr} \pl_r f \frac{\pl}{\pl r} \\
&= \kappa r f \frac{\pl}{\pl r}.\\
D_\xi X &= D_\xi (rf \frac{\pl}{\pl r})\\
&= \xi^r f \frac{\pl}{\pl r} + \frac{\kappa r^2}{f} \xi^r \frac{\pl}{\pl r} + rf \left( \xi^r D_\frac{\pl}{\pl r} \frac{\pl}{\pl r} + \xi^i D_\frac{\pl}{\pl \theta^i} \frac{\pl}{\pl r}  \right)\\
&= f \left( \xi^r \frac{\pl}{\pl r} + \xi^i \frac{\pl}{\pl \theta^i}  \right).
\end{align*}\end{proof}

On the hyperboloid, besides the obvious rotational symmetry, there are translational symmetries coming from the isometric action of the off-diagonal part of $O(1,n).$ We take a curve tangent to the hyperboloid:
\begin{align}\label{curve of isometric action}
\begin{pmatrix}
1 & t &  0 &      & 0 & \\
t & 1 &  0 &      &   & \\
0& 0 & 1&  &   & \\
 &   &     &\ddots   &   & \\
 0 &      &    &   &1 
\end{pmatrix}
\begin{pmatrix}
x^0\\
x^1\\
\vdots\\
\\
x^n
\end{pmatrix},
\end{align}
and project the curve onto $\mathbb{R}^n.$ The tangent vector of this curve gives a Killing vector field $\frac{f}{\sqrt{\kappa}} \frac{\pl}{\pl x^1}$ at $(x^1, x^2, \cdots, x^n).$ Hence any translational Killing vector field in the static coordinates is of the form $f Z_0 $ for some constant vector field $Z_0 = \sum_{i=1}^n a^i \frac{\pl}{\pl x^i}.$

On Euclidean space, the constant vector field is characterized by $dZ=0.$ On hyperbolic space, we introduce a twist covariant derivative in order to get a similar characterization.

\begin{defn}
The twist covariant derivative on $\mathbb{H}^n_{-\kappa}$ is defined as 
$$
\tilde{D}_\xi Z = D_\xi Z + \frac{\kappa}{f} \langle \xi,Z \rangle X,
$$
where $\langle\cdot,\cdot \rangle = g(\cdot,\cdot)$.
\end{defn}

\begin{prop}\label{const vec}
The constant vector field in hyperbolic space is characterized by $\tilde{D} Z=0$. I.e., $\tilde{D} Z=0$ if and only if $Z = \sum a_i \frac{\pl}{\pl x^i}$ in the static coordinates.
\end{prop}
\begin{proof}
We work in the hyperboloid model. From (\ref{curve of isometric action}), every constant vector is a linear combination of $\{ \frac{x^i}{x^0} \frac{\pl}{\pl x^0} + \frac{\pl}{\pl x^i} \}_{i=1}^n.$ Moreover, the tangent space is spanned by translational Killing vectors $\{ x^i \frac{\pl}{\pl x^0} + x^0 \frac{\pl}{\pl x^i} \}_{i=1}^n.$ Therefore, it suffices to check 
\begin{align*}
& \tilde{D}_{x^j \frac{\pl}{\pl x^0} + x^0 \frac{\pl}{\pl x^j}} \left(\frac{x^i}{x^0} \frac{\pl}{\pl x^0} + \frac{\pl}{\pl x^i} \right) \\
&= \left[ \left(-\frac{x^jx^i}{(x^0)^2} + \delta^{ij} \right) \frac{\pl}{\pl x^0} \right]^T  + \frac{\sqrt{\kappa}}{x^0} \left( -\frac{x^jx^i}{x^0} + x_0 \delta^{ij} \right) X \\
&= \left( -\frac{x^jx^i}{(x^0)^2} + \delta^{ij} \right) \frac{\pl}{\pl x^0} - \kappa \left( -\frac{x^jx^i}{x^0} + x_0 \delta^{ij} \right) \left( x^0 \frac{\pl}{\pl x^0} + \sum_{i=1}^n x^i \frac{\pl}{\pl x^i} \right) \\
& \quad + \kappa \left( -\frac{x^jx^i}{(x^0)^2} + \delta^{ij} \right) \left( \left( (x^0)^2 - \frac{1}{\kappa} \right) \frac{\pl}{\pl x^0} + x^0 \sum_{i=1}^n x^i \frac{\pl}{\pl x^i} \right) \\
&=0,
\end{align*}
where the superscript $T$ denotes the tangent component of a vector.
On the other hand, given a vector field $Z = \sum_{i=1}^n a_i(x) \left( \frac{x^i}{x^0} \frac{\pl}{\pl x^0} + \frac{\pl}{\pl x^i}\right)$,
\begin{align*}
\tilde{D}_\xi Z = \sum_{i=1}^n \xi(a_i) \left( \frac{x^i}{x^0} \frac{\pl}{\pl x^0} + \frac{\pl}{\pl x^i}\right).
\end{align*}
Therefore $\tilde{D} Z=0$ if and only if $Z$ is a constant vector field.
\end{proof}

\section{The infinitesimal rigidity}\label{rigid}

In this section, we prove the infinitesimal rigidity of convex surfaces, which will be used in the openness part of Theorem A. 

We start with some general discussion on infinitesimal deformations. Let $\mf{r}:(N, \sigma) \rightarrow (M,g)$ be an isometric embedding of a hypersurface and $\Sigma$ be the image of the embedding. Let $E = \mr^{-1} (TM)$ be the pull-back of the tangent bundle over $N$. We abuse notation to denote the pull-back connection by $D$. Denote the differential of $\mr$ by $D\mr$ and view it as a section $D \mf{r} \in \Gamma(N, E \otimes T^*N)$. We use the Einstein summation convention of summing repeated indices. The indices $a,b,c=1,2,\cdots, n-1$ and $i,j,\ldots=1,2,\cdots,n.$ Let $\{ u^a \}$ and $\{ x^i \}$ be local coordinates on $N$ and $M$, respectively. Let $\{e_a, \nu \}$ denote a local orthonormal frame such that $\{ e_a\}$ are tangent to $N$ and $\nu$ is the unit outward normal. Let $\{ \omega^a \}$ be the dual 1-form of $\{ e_a\}$.

\begin{defn}
For $v_1 \otimes \omega_1, v_2 \otimes \omega_2 \in \Gamma(N, E \otimes T^*N),$ we define 
$$
(v_1 \otimes \omega_1) \odot (v_2 \otimes \omega_2) :=\frac{1}{2} g(v_1,v_2) (\omega_1 \otimes \omega_2 + \omega_2 \otimes \omega_1)
$$ 
and extend it linearly.
\end{defn}

Suppose there is a family of isometric embeddings $\mf{r}_t$ with $\mf{r}_0(x) = \mf{r}(x).$ We say that $\mf{r}_t$ yields a first order isometric deformation of $\mf{r}$ if 
$$
\frac{d}{dt} \lt( D \mf{r}_t \odot D \mf{r}_t \rt) \bigg|_{t=0}=0.
$$ 
Set $\tau = \frac{d \mr_t}{dt} \big|_{t=0}$.  We have the following equivalent equation.

\begin{lem}
The infinitesimal deformation equation is given by
\begin{align}
D \mf{r} \odot D \tau=0, \label{infinitesimal deformation}
\end{align} 
where $\odot$ is the symmetric product of one-forms. We call $\tau$ an infinitesimal deformation (or an isometric deformation.) 
\end{lem}
\begin{proof}
We use  $\{u^a\}$ and  $\{x^i\}$ to denote local coordinates on $N$ and  $M$, respectively. Then 
$$
D \mr = \frac{\pl \mr^i}{\pl u^a}  \frac{\pl}{\pl x^i} \otimes du^a, \quad \tau=\frac{\pl \mf{r}_t^i}{\pl t}\bigg|_{t=0}  \frac{\pl}{\pl x^i},
$$ 
and
\begin{align*}
\sigma_{ab} du^a \odot du^b &= D\mr_t \odot D\mr_t(p) \\
&= \frac{\pl \mr_t^i}{\pl u^a}(p) \frac{\pl \mr_t^j}{\pl u^b} g_{ij}(\mr_t(p)) du^a \odot du^b.
\end{align*}
Differentiating with respect to $t$ and evaluating at $t=0$, we obtain
\begin{align*}
0 &= \lt( \frac{\pl^2 \mr^i}{\pl u^a \pl t} \frac{\pl \mr^j}{\pl u^b} g_{ij} + \frac{\pl \mr^i}{\pl u^a} \frac{\pl^2 \mr^j}{\pl u^b \pl t}g_{ij} + \frac{\pl \mr ^i}{\pl u^a}\frac{\pl \mr^j}{\pl u^b}\frac{\pl g_{ij}}{\pl x^k}\frac{\pl \mr^k}{\pl t} \rt) \bigg| _{t=0} du^a \odot du^b \\
&= \lt( \frac{\pl \tau^i}{\pl u^a}\frac{\pl \mr^j}{\pl u^b}g_{ij} + \frac{\pl \mr^i}{\pl u^a}\frac{\pl \tau^j}{\pl u^b}g_{ij} + \frac{\pl \mr^i}{\pl u^a}\frac{\pl \mr^j}{\pl u^b}\frac{\pl g_{ij}}{\pl x^k}\tau^k \rt) du^a \odot du^b.
\end{align*}
On the other hand,
\begin{align*}
D \tau = \lt( \frac{\pl \tau^j}{\pl u^b} + \frac{\pl \mr^m}{\pl u^b} \Gamma^j_{mk} \tau^k \rt) \frac{\pl}{\pl x^j} \otimes du^b.
\end{align*}
We then have
\begin{align*}
D \mr \odot D \tau 
&={1\over 2} \lt( \frac{\pl \mr^i}{\pl u^a} \lt( \frac{\pl \tau^j}{\pl u^b} + \frac{\pl \mr^m}{\pl u^b} \Gamma^j_{mk} \tau^k \rt)g_{ij} + (a,b \mbox{ symmetric}) \rt) du^a \odot du^b \\
&={1\over 2} \Big( \frac{\pl \tau^i}{\pl u^a}\frac{\pl \mr^j}{\pl u^b}g_{ij} + \frac{\pl \mr^i}{\pl u^a}\frac{\pl \tau^j}{\pl u^b}g_{ij}+ \frac{\pl \mr^i}{\pl u^a}\frac{\pl \mr^m}{\pl u^b} \left(\Gamma_{mk}^jg_{ij}+\Gamma_{ik}^jg_{mj}\right)\tau^k\Big) du^a \odot du^b  \\
&={1\over 2} \Big( \frac{\pl \tau^i}{\pl u^a}\frac{\pl \mr^j}{\pl u^b}g_{ij} + \frac{\pl \mr^i}{\pl u^a}\frac{\pl \tau^j}{\pl u^b}g_{ij}+ \frac{\pl \mr^i}{\pl u^a}\frac{\pl \mr^m}{\pl u^b} \frac{\pl g_{mi}}{\pl x^k}\tau^k\Big) du^a \odot du^b  \\
&=0,
\end{align*}
where in the second to last equality $i,m$ are symmetric as $a,b$ are symmetrized.
\end{proof}

\begin{defn}
A  solution of (\ref{infinitesimal deformation}) is called an \textit{infinitesimal deformation}. An infinitesimal deformation is \textit{trivial} if it is the restriction of some Killing vector field of $M$ on the hypersurface. An isometric embedding is called \textit{infinitesimally rigid} if it only has trivial infinitesimal deformations. 
\end{defn}

We now focus on the isometric embedding into hyperbolic space. Take $N=S^2$ and $M=\mathbb{H}^3_{-\kappa}.$ The ``position vector" $\mr$ can be replaced by the conformal Killing vector $X$ up to the static potential, which is nonzero. We henceforth write the infinitesimal deformation equation as
\begin{align}
D X \odot D \tau =0. \label{inf def} 
\end{align} 
The main result in this section is
\begin{thm}{\label{main}}
Consider an isometric embedding $\mr: (S^2, \sigma) \rw (\mathbb{H}^3_{-\kappa},g). $ Suppose that the image $\Sigma$ is  a convex hypersurface in $\mathbb{H}^3_{-\kappa}.$ Then $\mr$ is infinitesimally rigid. 
\end{thm}

Before proving this theorem we need to establish some basic results. The fiber of $E$ is a 3-dimensional inner product space. We can define the usual cross product $\times$ once we fix an orientation. The identity $u \times (v \times w) = (u \cdot w)v - (u \cdot v)w$ will be useful. First we need a definition.

\begin{defn}
We define the inner product and cross product for differential forms valued in $E\otimes T^*\mathbb{H}^3_{-\kappa}$ by
\begin{align*}
\langle v_1 \otimes \omega_1,v_2 \otimes \omega_2 \rangle = \langle v_1, v_2 \rangle \otimes (\omega_1 \wedge \omega_2),\\
(v_1 \otimes \omega_1) \times (v_2 \otimes \omega_2) := (v_1 \times v_2) \otimes (\omega_1 \wedge \omega_2),
\end{align*}
and extend it linearly. We also have the pull-back twisted covariant derivative $\tilde{D}: \Gamma(S^2, E) \rightarrow \Gamma(S^2, E \otimes T^*S^2)$ defined as
\begin{align*}
\tilde{D}_\xi Z := D_\xi Z + \frac{\kappa}{f} \langle \mr_*(\xi),Z \rangle X
\end{align*}
for $ \xi\in TS^2$ and $Z\in \Gamma(S^2, E).$
\end{defn}

Next, we need the following lemmas. 

\begin{lem}\label{rotation vector}
Suppose $\tau$ is an infinitesimal deformation. Then 
$$
\tilde{D} \lt( \frac{\tau}{f} \rt) = Y\times D X
$$
for some $Y,$ which is called the rotation vector of $\tau.$ 
\end{lem}
\begin{proof}
Choose a local orthonormal frame $\{e_1,e_2, \nu\}$. Let $\omega^1, \omega^2$ be the dual of $e_1,e_2$. By (\ref{grad f}), we have
\begin{align*}
\tilde{D}(\frac{\tau}{f}) &= \lt[ D_{e_a} (\frac{\tau}{f}) + \frac{\kappa}{f} \langle e_a,\frac{\tau}{f} \rangle X \rt] \omega^a \\
&= \lt[ \frac{D_{e_a}\tau}{f} + e_a(\frac{1}{f})\tau + \frac{\kappa}{f^2}\langle e_a,\tau \rangle \langle X,e_i \rangle e_i \rt]\omega^a \\
&= \lt[ \frac{D_{e_a}\tau}{f} - \frac{\kappa}{f^2} \langle X,e_a \rangle \langle \tau,e_i \rangle e_i + \frac{\kappa}{f^2} \langle \tau ,e_a \rangle \langle X,e_i \rangle e_i \rt]\omega^a \\
&= \lt[ A_{ab}e_b + B_a \nu \rt]\omega^a.
\end{align*}
From (\ref{inf def}),
\begin{align*}
D_{e_a}\tau \cdot e_b + D_{e_b}\tau \cdot e_a =0.
\end{align*}
Hence $A_{ab}$ is anti-symmetric and $\tilde{D}(\frac{\tau}{f}) = Y \times D X$ for  $Y=\frac{1}{f}(B_2 e_1 - B_1 e_2 + A_{12}\nu).$
\end{proof}

When we restrict a Killing vector field of hyperbolic space to the surface, we get an infinitesimal deformation $\tau = Y_0 \times X + f Z_0$. We compute the rotation vector for such $\tau$.
\begin{lem}\label{rotation vector of a Killing vector}
For a Killing vector $\tau = Y_0 \times X + f Z_0$, where $Y_0$ and $Z_0$ are constant vectors, its rotation vector is $\frac{1}{f^3} \lt( Y_0 + \kappa \langle Y_0,X \rangle X \rt)$.
\end{lem}
\begin{proof}
By Proposition \ref{const vec}, $\tilde{D}Y_0=0$ and $\tilde{D}Z_0=0$. Thus,
\begin{align*} 
\tilde{D} \lt( \frac{\tau}{f} \rt) 
&= \frac{Y_0}{f} \times D X + \lt( (Y_0 \times X) e_a(\frac{1}{f}) + \frac{\kappa}{f} \langle e_a,\frac{Y_0 \times X}{f}\rangle X \rt) \omega^a\\
&= \frac{Y_0}{f} \times D X + \frac{\kappa}{f^3} \left(\langle DX, Y_0 \times X \rangle X - \langle D X,X \rangle (Y_0 \times X) \right)\\
&= \frac{Y_0}{f} \times D X + \frac{\kappa}{f^3}\left((Y_0 \times X) \times X \right) \times D X.
\end{align*}
This implies that
\begin{align*}
Y&= \frac{Y_0}{f} + \frac{\kappa}{f^3}(Y_0 \times X) \times X \notag\\
 &= \frac{Y_0}{f} +\frac{\kappa}{f^3} \lt( \langle Y_0,X \rangle X - \langle X,X \rangle Y_0 \rt) \notag\\
 &= \frac{Y_0}{f}(1-\frac{\kappa r^2}{f^2}) + \frac{\kappa}{f^3} \langle Y_0,X \rangle X\notag\\
 &= \frac{1}{f^3} \lt( Y_0 + \kappa \langle Y_0,X \rangle X \rt).  
\end{align*}
\end{proof}

\begin{lem} \label{characterization}
A vector $Z$ can be written as $Z_0 + \kappa \langle Z_0,X \rangle X$ for some constant vector $Z_0$ if and only if 
\begin{align*}
D Z = \frac{\kappa}{f^2} \langle Z,X \rangle D X.
\end{align*}
\end{lem}

\begin{proof}
The only if part follows by direct computation. For the if part, let $W= D_\xi Z - \frac{\kappa}{f}\langle Z,X \rangle \xi =0.$ We then compute
\begin{align*}
\tilde{D}_\xi \lt( Z- \frac{\kappa}{f^2} \langle Z,X \rangle X \rt) = W - \frac{\kappa}{f^2} \langle W,X \rangle X =0.
\end{align*}
Hence $Z = \frac{\kappa}{f^2} \langle Z,X \rangle X + Z_0$ for some constant vector $Z_0.$ The assertion follows by taking inner product with $X.$
\end{proof}

In terms of the cross product, we express the Riemann curvature tensor on hyperbolic space as 
\begin{align*}
R(X,Y)Z = -\kappa \langle Y,Z \rangle X + \kappa \langle X,Z \rangle Y = \kappa Z \times (Y \times X).
\end{align*}

%We denote $\frac{\pl \mr}{\pl u^i}$ by $\mr_i$ ,$D_{\mr_i}W$ by $D_i W.$%

Let $\bY = f^3 Y,$ where $Y$ is the rotation vector of $\tau$. 
\begin{lem} \label{bY tangential}
$D \bY$ is tangential, that is, $D \bY = C_a^b e_b \omega^a$ for some $2\times 2$ matrix $C_a^b.$ Moreover, $2\kappa \langle \bY,X \rangle = f(C^1_1 + C^2_2)$.
\end{lem}
\begin{proof}
From Lemma \ref{rotation vector}, $\tilde{D}( \frac{\tau}{f} ) = D ( \frac{\tau}{f}) + \frac{\kappa}{f} \langle \frac{\tau}{f},\cdot \rangle X = Y \times D X.$ Choosing an orthonormal frame $e_1, e_2$ such that the tangential component $D_{e_a}^T e_b(p)=0$. Multiplying by $f^3$ and taking derivative, we have at $p$, 
\begin{align*}
f^3 D_{e_b}D_{e_a} \frac{\tau}{f} + e_b(f^3)D_{e_a} \frac{\tau}{f} + D_{e_b} \lt( \kappa f \langle \tau,e_a \rangle X \rt) = \lt( D_{e_b} \bY \rt) \times D_{e_a} X + \bY \times D_{e_b}D_{e_a} X.
\end{align*}
Antisymmetrizing $a,b$ and using the curvature identity, we get
\begin{align*}
&\quad f^3 \lt( -\kappa \langle e_a,\frac{\tau}{f} \rangle e_b + \kappa \langle e_b,\frac{\tau}{f} \rangle  e_a \rt) + e_b(f^3) D_{e_a} (\frac{\tau}{f}) - e_a(f^3) D_{e_b} (\frac{\tau}{f}) \\
&\qquad\qquad+ e_b(\kappa f \langle \tau,e_a \rangle )X - e_a(\kappa f \langle \tau,e_b \rangle )X + \kappa f \big(\langle \tau,e_a \rangle D_{e_b} X - \langle \tau,e_b \rangle D_{e_a} X \big)\\
&= D_{e_b}\bY \times D_{e_a} X - D_{e_a}\bY \times D_{e_b} X + \bY \times \lt( -\kappa \langle e_a,X \rangle e_b + \kappa \langle e_b,X \rangle e_a \rt).
\end{align*}
Note the first and last term of the left hand side cancel. By Lemma \ref{rotation vector}, (\ref{conformal Killing}), and (\ref{grad f}),
\begin{align*}
&\quad e_b(f^3) Y \times D_{e_a} X - e_a(f^3) Y \times D_{e_b} X + \kappa f^2 \lt( \langle Y \times D_{e_b} X,e_a \rangle - \langle Y \times D_{e_a} X,e_b \rangle \rt) X\\
&= D_{e_b} \bY \times D_{e_a} X - D_{e_a} \bY \times D_{e_b} X -\kappa f^3 \lt( \langle e_a,X \rangle Y\times e_b - \langle e_b,X \rangle Y\times e_a \rt).
\end{align*}
Therefore,
\begin{align*}
&\quad  D_{e_b}\bY \times D_{e_a}X - D_{e_a} \bY \times D_{e_b}X\\
&=2\kappa f^3 \lt( \langle X,e_b \rangle Y \times e_a - \langle X,e_a \rangle Y \times e_b \rt) + \kappa f^3 \lt( \langle Y \times e_b,e_a \rangle - \langle Y \times e_a,e_b \rangle \rt)X  \\
&= 2\kappa \bar{Y} \times \left(X\times \left(e_a \times e_b\right)\right)-2\kappa \langle \bar{Y}, (e_a \times e_b) \rangle X.
\end{align*}
Let $a=1,b=2$. We have
\begin{align*}
D_{e_2} \bY \times fe_1 - D_{e_1} \bY \times fe_2 = 2\kappa \bY \times (X \times \nu) - 2\kappa \langle \bY,\nu \rangle X = -2\kappa \langle \bY,X \rangle \nu.
\end{align*}
The lemma follows by comparing the tangential and normal components.
\end{proof}

\begin{lem}\label{Psi}
Consider
\begin{align*}
\Psi= d \lt[ \frac{1}{f}X \cdot \rt( \bY - \frac{\kappa}{f^2} \langle \bY,X \rangle X \lt) \times \lt( D \bY - \frac{\kappa}{f^2} \langle \bY,X \rangle D X \rt)\rt].
\end{align*} 
Here $v \cdot w$ stands for $\langle v,w \rangle$. The two-form $\Psi$ can be expressed as
$$
\frac{X}{f} \cdot \lt( D\bY - \frac{\kappa}{f^2}\langle \bY,X \rangle D X \rt) \times \lt( D\bY - \frac{\kappa}{f^2} \langle \bY,X \rangle D X \rt).
$$
\end{lem}
\begin{proof}
By straightforward computation, we have
\begin{align*}
f^2 \Psi &= -df X \cdot \bY \times \lt( D \bY - \frac{\kappa}{f^2} \langle \bY,X \rangle D X \rt)\\
&\quad + fD X \cdot \lt( \bY - \frac{\kappa}{f^2} \langle \bY,X \rangle X \rt) \times \lt( D \bY - \frac{\kappa}{f^2} \langle \bY,X \rangle D X \rt)\\
&\quad + fX \cdot \lt( D\bY - \frac{\kappa}{f^2}\langle \bY,X \rangle D X \rt) \times \lt( D\bY - \frac{\kappa}{f^2} \langle \bY,X \rangle D X \rt)\\
&\quad + f X \cdot \bY \times \lt( D^2 \bY -d \lt( \frac{\kappa}{f^2}\langle \bY,X \rangle \rt) \wedge D X - \frac{\kappa}{f^2}\langle \bY,X \rangle D^2 X \rt)\\
&= I + II + III + IV.
\end{align*}
Here 
\begin{align*}
D^2 Y &= \frac{1}{2} \left(  D_{e_a} D_{e_b} \bY - D_{e_b}D_{e_a} \bY - D_{[e_a,e_b]} \bY \right) \omega^a \omega^b \\
&= -\kappa (\bY \times \nu) \omega^1 \wedge \omega^2.
\end{align*}
by the curvature identity.

We claim that $II=0.$ Indeed, we have 
\begin{align}
D X \times D \bY &= f(C^1_1 + C^2_2)\nu \omega^1 \wedge \omega^2 \label{DX times DbY}\\
D X \times D X &= 2 f^2 \nu \omega^1 \wedge \omega^2. \label{DX times DX}
\end{align}
By Lemma \ref{bY tangential}, $II=0$. Consequently, it remains to show that $I + IV =0$

The following identity can be verified directly by expanding each term in an orthonormal frame. 
\begin{align}
v \times (\alpha \times \beta) = -\langle v,\alpha \rangle \beta - \langle v,\beta \rangle \alpha \label{v times (alpha times beta)}
\end{align}
for any $v \in \Gamma(S^2,E)$ and $\alpha,\beta \in \Gamma(S^2, E \otimes T^*S^2)$. 

From (\ref{grad f}), we have
\begin{align*}
df &= \frac{\kappa}{f} \langle X,DX \rangle \\
d \left( \frac{\kappa}{f^2} \langle \bY,X \rangle \right) &= \frac{\kappa}{f^2} \left( -\frac{\kappa}{f^2} \langle X,DX \rangle \langle \bY,X \rangle + \langle D \bY,X \rangle + \langle \bY,DX \rangle \right)
\end{align*}
Hence $I + IV = \frac{\kappa}{f} X \cdot \bY \times \Omega$ where
\begin{align*}
\Omega & = \left( - \langle X,DX \rangle \wedge D\bY + \frac{\kappa}{f^2} \langle \bY,X \rangle \langle X,DX \rangle \wedge DX \right.\\
&\qquad -f^2(\bY \times \nu) \omega^1 \wedge \omega^2 \\
&\qquad + \frac{2\kappa}{f^2} \langle \bY,X \rangle\langle X,DX \rangle \wedge DX - \langle X,D\bY \rangle \wedge DX - \langle \bY,DX \rangle \wedge DX \\
&\qquad + \kappa \langle \bY,X \rangle (X \times \nu) \omega^1 \wedge \omega^2 \Big)
\end{align*}
From (\ref{v times (alpha times beta)}), (\ref{DX times DX}) and (\ref{DX times DbY}), we obtain
\begin{align*}
\Omega &= \Big( 2\kappa \langle \bY,X \rangle (X\times \nu) - \kappa \langle \bY,X \rangle (X\times\nu) - f^2 (\bY\times\nu) \\
&\qquad - 2\kappa \langle \bY,X \rangle (X\times\nu) + f^2(\bY\times\nu) + \kappa \langle \bY,X \rangle (X\times\nu) \Big) \omega^1 \wedge \omega^2 \\
&=0 
\end{align*}
\end{proof}

We are ready to prove the rigidity result, Theorem \ref{main}.

\begin{proof}[Proof of Theorem 7]
From Lemma \ref{bY tangential}, we can write 
\begin{equation}
D \bY - \frac{\kappa}{f^2} \langle \bY,X \rangle D X = B^a_b e_a \omega^b \label{B}
\end{equation}
with $B^1_1 + B^2_2 =0.$ Choosing an orthonormal frame $e_1, e_2$ such that the tangential component $D_{e_a}^T e_b(p)=0$. At $p,$ we have
\begin{align*}
D_{e_a} \lt( D_{e_b} \bY - \frac{\kappa}{f^2} \langle \bY,X \rangle D_{e_b}X \rt) = e_a(B^c_b)e_c - B^c_b h_{ac} \nu,
\end{align*}
where $h_{ab}$ is the second fundamental form of $\Sigma.$
Antisymmetrizing and using the curvature identity, we obtain
\begin{eqnarray*}
&-\kappa \bY \times (e_a \times e_b) + \alpha e_b - \beta e_a + \frac{\kappa^2}{f^2} \langle \bY,X \rangle X \times (e_a \times e_b) \\
&= \lt( e_a(B^c_b) - e_b(B^c_a) \rt)e_c + \lt( B^c_a h_{bc} - B^c_b h_{ac} \rt) \nu.
\end{eqnarray*}
Here $\alpha= -fD_{e_a}\lt(  \frac{\kappa}{f^2} \langle \bY,X \rangle\rt)$ and $\beta = -fD_{e_b}\lt(  \frac{\kappa}{f^2} \langle \bY,X \rangle\rt)$.
Let $a=1,b=2.$ Comparing the normal component, we have $B^c_1 h_{2c} - B^c_2 h_{1c}=0.$ In matrix form,
\begin{align} \label{negative definite}
Tr \begin{pmatrix}
B^1_1 & B^2_1 \\
B^1_2 & B^2_2 \\
\end{pmatrix}
\begin{pmatrix}
h_{21} & -h_{11} \\
h_{22} & -h_{12} \\
\end{pmatrix}=0.
\end{align}
By the Gauss equation, the determinant of the second matrix is $K + \kappa$. Since $(Tr(M))^2 \geq 2 \det(M)$ for any $2 \times 2$ matrix, $\det(B) \leq 0.$ On the other hand, Lemma \ref{Psi} and Stoke's theorem imply
\begin{align*}
0 &= \int_{\Sigma} \Psi \\
 &= \int_{\Sigma}  d \lt[ \frac{1}{f}X \cdot \rt( \bY - \frac{\kappa}{f^2} \langle \bY,X \rangle X \lt) \times \lt( D \bY - \frac{\kappa}{f^2} \langle \bY,X \rangle D X \rt)\rt] \\
   &= \int_{\Sigma} \frac{X}{f}\cdot \left( B_b^a e_a \omega^b \right) \times \left( B_d^c e_c \omega^d \right) \\
  &= \int_{\Sigma} 2\frac{\langle X,\nu \rangle}{f} \det B \omega^1\wedge \omega^2.
\end{align*}
Hence, $\det(B) \equiv 0.$ We may assume $h_{11}h_{22} > 0, h_{12}=h_{21}=0.$ (\ref{negative definite}) then implies $B^1_2 B^2_1 \geq 0.$ Since $B^1_1 + B^2_2=0, B \equiv 0.$

By Lemma \ref{characterization}, $\bY = Z_0 + \kappa \langle Z_0,X \rangle X$ for some constant vector $Z_0.$ From the definition of $\bY,$ $\tilde{D} \lt( \frac{\tau}{f} \rt) = \frac{1}{f^3} \lt( Z_0 + \kappa \langle Z_0,X \rangle X \rt) \times D X.$ By Lemma \ref{rotation vector of a Killing vector}, $\tilde{D} \lt( \frac{\tau - Z_0 \times X}{f} \rt)=0.$ Hence $\tau = Z_0 \times X + f W_0$ for some constant vector $W_0.$ This completes the proof of Theorem \ref{main}.
\end{proof}

\section{Isometric embeddings into $\mathbb{H}^3_{-\kappa}$}

The main goal of this section is to prove
\begin{thm}[Theorem A]\label{isom}
Let $\bar{\sigma}$ be a smooth metric on $S^2$ with Gauss curvature $K \geq-\kappa$. Then there exists a $C^{1,1}$ isometric embedding into $\mathbb{H}^3_{-\kappa}$. 
\end{thm}

We prove the theorem by the continuity method. It consists of three steps:
\begin{itemize}
\item[(1)] Connectedness: show that there exists a family of smooth metrics $\sigma_t, t\in [0,\infty)$ such that $\sigma_0 = \bar{\sigma}$ and that $\sigma_t$ converges to a metric with constant Gauss curvature. Moreover, $\sigma_t$  has $K > -\kappa$ for $t>0$.

\end{itemize}
Let $I \subset [0,\infty)$ be the set of parameters that $\sigma_t $ can be isometrically embedded into  $\mathbb{H}^3_{-\kappa}$ as a closed convex $C^{1,1}$ surface.

\begin{itemize}
\item[(2)] Openness: show that $I \cap (0,\infty)$ is open.
\end{itemize}
\noindent Since any metric with constant Gauss curvature is the same as the standard metric up to a diffeomorphism, $\sigma_\infty$ can be isometrically embedded into $\mathbb{H}^3_{-\kappa}.$ As a result of the openness, there exists a sufficiently large $T_0$ such that $\sigma_t, t>T_0,$ can be isometrically embedded into $\mathbb{H}^3_{-\kappa}$. In particular, $I$ is nonempty.
\begin{itemize}
\item[(3)] Closedness: We prove an a priori estimate to obtain closedness. 
\end{itemize}

%explain more
\subsection{Connectedness}

We prove the connectedness by using solutions to the normalized Ricci flow. 

\begin{lem}\label{Ricci flow on surfaces}
There exists an one-parameter family of smooth metrics $\sigma_t, t \in [0,\infty)$ on $S^2$ such  that $\sigma_0 = \bar{\sigma}$ and $\sigma_t$ converges to a metric $\sigma_\infty$ with constant Gauss curvature. Moreover, $K > -\kappa$ for $t >0.$  
\end{lem}

\begin{proof}
For any given smooth metric $\bar{\sigma}$ on $S^2$ with $K \geq -\kappa$, consider the normalized Ricci flow with the initial metric $\bar{\sigma}:$
\begin{align}\label{normalized Ricci flow}
\left\{ \begin{array}{rl}
\frac{\pl \sigma}{\pl t} &= (r- R)\sigma\\
\sigma_0 &= \bar{\sigma},
\end{array} \right.
\end{align}
where $R$ is the scalar curvature of $\sigma_t$ with average $r.$  Hamilton \cite{Hamilton88} and Chow \cite{Chow91} established long time existence and convergence to a metric with constant Gauss curvature. Moreover, the scalar curvature satisfies the evolution equation
\begin{align*}
\frac{\partial R}{\partial t} = \Delta_{\sigma}R + R(R-r).
\end{align*}
Applying the maximum principle to the evolution equation, we know that $\min R$ is increasing when $\min R <0$. Thus, $K(\sigma_t)+ \kappa$ is positive for $t>0$.  
\end{proof}

\subsection{Openness}
We show the openness in the continuity method by proving the following result. 
\begin{thm}\label{openness}
Let $\sigma$ be a smooth metric on $S^2$ with Gauss curvature  $K > -\kappa$. Suppose $\sigma$ can be isometrically embedded into $\mathbb{H}^3_{-\kappa}$ as a closed convex surface $\mr$. Then for any $\alpha \in (0,1),$ there exists a positive $\epsilon,$ depending only on $\sigma$ and $\alpha,$ such that any smooth metric $\sigma'$ on $S^2$ satisfying
\begin{align*}
|\sigma - \sigma'|_{C^{2,\alpha}} < \epsilon
\end{align*}
can be isometrically embedded in $\mathbb{H}^3_{-\kappa}$ as a closed convex survace $\mr'$.
\end{thm}

Recall that we write $\mr$ for $\varphi \circ i$ and view $\mr$ as the position vector, where $\varphi$ is a trivialization. A deformation $\tau \in \Gamma(\Sigma,E)$ is viewed as a vector-valued function $\my \in T_\mr \mathbb{R}^3$ and a translation sending $\mr$ to $\mr + \my.$ To prove Theorem \ref{openness}, it suffices to find a vector $\my$ satisfying
\begin{align}
g_{ij}({\mr+\my}) \frac{\pl(\mr^i+\my^i)}{\pl u^a} \frac{\pl(\mr^j+\my^j)}{\pl u^b} = \sigma'_{ab}. \label{isometric embedding equation for nearby metric}
\end{align}
Below, we first find an equivalent infinitesimal-deformation equation (\ref{linear}) and solve the linearized equation (\ref{linearized eq}) of (\ref{linear}). 

Substracting $g_{ij} (\mr) \frac{\pl \mr^i}{\pl u^a} \frac{\pl \mr^j}{\pl u^b} = \sigma_{ab}$ from (\ref{isometric embedding equation for nearby metric}), we get
\begin{align*}
g_{ij}(\mr) \left( \mr_a^i D_b \my^j + D_a \my^i \mr_b^j \right) - g_{ij}(\mr)\left( \mr_a^i \Gamma_{bk}^j \my^k +  \Gamma_{ak}^i \my^k \mr_b ^j \right)+ g_{ij}(\mr) \my_a^i \my_b^j \\
+ \left( g_{ij}(\mr+\my) - g_{ij}(\mr) \right) \left( \mr_a^i \mr_b^j + \mr_a^i \my_b^j + \my_a^i \mr_b^j + \my_a^i \my_b^j \right) = \sigma'_{ab} - \sigma_{ab},
\end{align*}
where $D_b = D_{\frac{\pl}{\pl u^b}}$, $\mr^i_a = \frac{\pl \mr^i}{\pl u^a}$ and $\Gamma_{ak}^i = \mr_a^m \Gamma_{mk}^i.$
By Taylor theorem,
\begin{align*}
g_{ij}(\mr+\my) - g_{ij}(\mr) = \pl_k g_{ij}(\mr) \my^k + \my^k\my^l \int_0^1 (1-t) \pl^2_{kl} g_{ij}(\mr + t\my) dt.
\end{align*}
By the definition of Christoffel symbol, $\pl_k g_{ij}(\mr) \mr_a^i \mr_b^j = g_{ij}(\mr) \left(  \mr_a^i \Gamma_{bk}^j + \Gamma_{ak}^i \mr_b^j \right).$ 
% If we replace $\my$ by $\mz$ on the right-hand side of (\ref{isometric embedding equation}), 
Set 
$$
F_{ijkl} (\mr,\my) := \int_0^1 (1-t) \pl^2_{kl} g_{ij}(\mr + t\my) dt
\quad \mbox{and} \quad
G_{ijk} (\mr,\my) := \int _0^1 \pl_k g_{ij}(\mr + t\my) dt.
$$ 
 We conclude that (\ref{isometric embedding equation for nearby metric}) is equivalent to the following inhomogeneous infinitesimal-deformation-type equation
\begin{align}\label{linear}
&g_{ij}(\mr) \left( \mr_a^i D_b \my^j + D_a \my^i \mr_b^j \right) \\ 
&=  \sigma'_{ab} -\sigma_{ab} - g_{ij}(\mr) \my_a^i \my_b^j - F_{ijkl}(\mr,\my)\mr_a^i\mr_b^j \my^k \my^l - G_{ijk}(\mr,\my) \my^k (\mr_a^i \my_b^j + \my_a^i \mr_b^j + \my_a^i \my_b^j)\nonumber \\
& =: q_{ab}(\my) \notag
\end{align}
Note that $|F_{ijkl}(\mr,\my)|_{m,\alpha}, |G_{ijk}(\mr,\my)|_{m,\alpha} \leq C_{m,\alpha} |\my|_{m,\alpha}.$ 

To solve (\ref{linear}), we study the corresponding linearized equation
\begin{align}\label{linearized eq}
g_{ij}(\mr) \left( \mr_a^i D_b \my^j + D_a \my^i \mr_b^j \right) = \bar{q}_{ab}
\end{align}
where $\bar{q}_{ab}$ is an arbitrary smooth symmetric bilinear form on $S^2$.

Before solving the linearized equation (\ref{linearized eq}), we show that $\tilde{D}$ is a flat connection on $E.$
\begin{lem} \label{flat}
For any $Y\in \Gamma(\Sigma, E),$
\begin{align*}
\tilde{D}_{e_b} \tilde{D}_{e_a} Y - \tilde{D}_{e_a} \tilde{D}_{e_b} Y - \tilde{D}_{[e_b,e_a]} Y =0.
\end{align*}
\end{lem}

\begin{proof} 
\begin{align*} 
\tilde{D}_{e_b} \tilde{D}_{e_a} Y &= \tilde{D}_{e_b} \left( D_{e_a} Y + \frac{\kappa}{f} \langle e_a,Y \rangle X \right) \\
&= D_{e_b} D_{e_a} Y  + D_{e_b} \left( \frac{\kappa}{f} \langle e_a,Y \rangle X \right) + \frac{\kappa}{f} \langle e_b,D_{e_a} Y \rangle X+ \frac{\kappa^2}{f^2} \langle e_a,Y \rangle \langle e_b,X \rangle X\\
&=  D_{e_b} D_{e_a} Y  
 -\frac{\kappa^2}{f^2} \langle e_b,X \rangle\langle e_a,Y \rangle X + \frac{\kappa}{f}\langle D_{e_b} e_a,Y \rangle X + \frac{\kappa}{f} \langle e_a,D_{e_b} Y
\rangle X + \kappa \langle e_a,Y \rangle e_b\\
&+ \frac{\kappa}{f} \langle e_b,D_{e_a} Y \rangle X+ \frac{\kappa^2}{f^2} \langle e_a,Y \rangle \langle e_b,X \rangle X.
\end{align*}
Therefore,
\begin{align*}
\tilde{D}_{e_b} \tilde{D}_{e_a} Y - \tilde{D}_{e_a}\tilde{D}_{e_b} Y - \tilde{D}_{[e_b,e_a]} Y =R(e_b,e_a)Y + \kappa \langle e_a,Y \rangle e_b - \kappa \langle e_b,Y \rangle e_a =0.
\end{align*}
\end{proof}

Next, we solve the linearized equation (\ref{linearized eq}). 

\begin{prop}\label{linear prop}
For any smooth symmetric bilinear form $\bar{q}$ on a convex surface $\Sigma \subset \mathbb{H}^3_{-\kappa}$, there exists a smooth solution to 
\begin{align}
D \tau \odot DX = f^2 \bar{q} \label{inhomo inf def}.
\end{align}
\end{prop}

\begin{proof}
For a fixed point $p$, we choose an orthonormal frame $\{e_1,e_ 2\}$ with $D_{e_a}^T e_b(p)= 0$ and write $D_a$ for $D_{e_a}$ and $\bar{q}_{ab}$ for $\bar{q}(e_a,e_b)$. Equation (\ref{inhomo inf def}) implies the symmetric part of tangential components of $D\tau$ is $f\bar{q}_{ab},$ equivalently, 
\begin{align}
\frac{1}{2}\lt( \tilde{D}_a \lt( \frac{\tau}{f} \rt)\cdot e_b + \tilde{D}_b \lt( \frac{\tau}{f} \rt)\cdot e_a \rt)= \bar{q}_{ab}.
\end{align}
To solve (\ref{inhomo inf def}), we introduce new dependent variables. We define $v_1, v_2$ and $w$ by
\begin{align}
v_a = \tilde{D}_a \lt( \frac{\tau}{f} \rt)\cdot \nu, \quad a=1,2,
\end{align}
and
\begin{align}
\frac{w}{f^2} = \frac{1}{2}\lt( \tilde{D}_a \lt( \frac{\tau}{f} \rt)\cdot e_b - \tilde{D}_b \lt( \frac{\tau}{f} \rt)\cdot e_a \rt)
\end{align}

The triplet $\{v_1, v_2, w \}$ completely determines $\tilde{D} \lt( \frac{\tau}{f} \rt)$:
\begin{align}
\tilde{D}_a \left( \frac{\tau}{f}\right) = \sum_{b=1}^2 \left( \bar{q}_{ab} + \frac{w}{f^2} \epsilon_{ab} \right) e_b + v_a \nu \label{tau in terms of w and u}
\end{align}
where $\epsilon_{ab} = \begin{pmatrix} 0&1 \\-1&0 \end{pmatrix}$.

%After introducing the dependent variables, the following computation can be done intrinsically. Even the second fundamental form is regarded as a symmetric 2-tensor satisfying $\na_a h_{bc} =\na_b h_{ac}.$
Let $\na$ denote the Levi-Civita connection with respect to $\sigma.$ Since $\tilde{D}$ is a flat connection, we have 
\begin{align*}
0 &= \tilde{D}_a \tilde{D}_b \left( \frac{\tau}{f}\right) - \tilde{D}_b \tilde{D}_a \left( \frac{\tau}{f}\right)\\
&= \sum_{c=1}^2 \left[ \na_a \left( \bar{q}_{bc} + \frac{w}{f^2} \epsilon_{bc} \right) e_c - \lt( \bar{q}_{bc} + \frac{w}{f^2}\epsilon_{bc} \rt) h_{ac} \nu + v_b h_{ac} e_c \right] + \frac{\kappa}{f} \lt( \bar{q}_{ba} + \frac{w}{f^2}\epsilon_{ba} \rt)X \\
& \quad - \lt( a,b \quad antisymmetric \rt).
\end{align*}
Comparing the tangential and normal components, we have
\begin{align}
\lt( \begin{array}{cc}
h_{11} & h_{21}\\
h_{12} & h_{22}
\end{array} \rt)
\lt( \begin{array}{c}
v_2\\
-v_1
\end{array} \rt)
&=
\lt( \begin{array}{cc}
\frac{\na_1 w}{f^2} -c_1 \\
\frac{\na_2 w}{f^2} -c_2
\end{array} \rt) \label{umatrix},
\end{align}
and
\begin{align}
\na_1 v_2 - \na_2 v_1 &= T -\frac{Hw}{f^2} + \frac{2\kappa}{f^3} w \langle X,\nu \rangle, \label{Hodge du}
\end{align}
where $T=- \sum_{a=1}^2 (\bar{q}_{1a} h_{a2} + \bar{q}_{2a}h_{a1})$ and $c_a= \na_1 \bar{q}_{a2} - \na_2 \bar{q}_{a1}.$
We substitute (\ref{umatrix}) into (\ref{Hodge du}) to obtain an elliptic equation
\begin{align}
 \na_a \lt( (h^{-1})^{ab} \frac{\na_b w}{f^2} \rt) + \frac{Hw}{f^2} - \frac{2\kappa}{f} \langle X,\nu \rangle w = T + \na_a \lt( (h^{-1})^{ab}c_b \rt). \label{elliptic}
\end{align}
as $\det\lt( h_{ab} \rt) >0$.

In order to solve the self-adjoint elliptic equation (\ref{elliptic}), we need to show the quantity of the right-hand side is perpendicular to the kernel of the operator. The homogeneous equation associated with (\ref{elliptic}) corresponds to an infinitesimal isometric deformation. From the infinitesimal rigidity result in Section \ref{rigid}, the general solution comes from Killing vector fields $Y_0 \times X + fZ_0.$ 

For $\tau = Y_0 \times X + fZ_0$, by (\ref{rotation vector of a Killing vector}),  
\begin{align*}
w 
= \frac{1}{2}f^2 \lt( \tilde{D}_{e_1} \lt( \frac{\tau}{f} \rt)\cdot e_2 - \tilde{D}_{e_2} \lt( \frac{\tau}{f} \rt)\cdot e_1 \rt)
= \langle Y_0, \nu \rangle + \kappa \langle Y_0, X \rangle \langle X, \nu \rangle. 
\end{align*}
Hence the kernel is of the form
\begin{align*}
\bar{w} = \langle Y_0,\nu \rangle + \kappa \langle Y_0,X \rangle \langle X,\nu \rangle
\end{align*}
for any constant vector $Y_0$. 
To show 
\begin{align*}
\int T\bar{w} - (h^{-1})^{ab}c_b \na_a \bar{w} =0, 
\end{align*}
we introduce new quantities 
\begin{align*}
\bar{w}_a = \langle Y_0,e_a \rangle + \kappa \langle Y_0,X \rangle \langle X,e_a \rangle, a=1,2.
\end{align*} 
Since $\tilde{D}Y_0=0$, we have the following
\begin{align*}
\na_a \bar{w} = \sum_{b=1}^2 h_{ab} \lt( \langle Y_0,e_b \rangle + \kappa \langle Y_0,X \rangle \langle X,e_b \rangle \rt)= \sum_{b=1}^2 h_{ab} \bar{w}_b,
\end{align*}
and 
\begin{align*}
\na_a \bar{w}_b = -h_{ab}\bar{w} + \kappa f \langle Y_0,X \rangle \delta_{ab}.
\end{align*}
We are now ready to verify that 
\begin{align*}
\int T\bar{w} - (h^{-1})^{ab}c_b \na_a \bar{w} &= \int T\bar{w} -\sum_{a=1}^2 c_a \bar{w}_a \\
&= \int T\bar{w} + \sum_{a=1}^2 \left( \bar{q}_{a2} \na_1 \bar{w}_a - \bar{q}_{a1} \na_2 \bar{w}_a \right) \\
&= \int T \bar{w} - \sum_{a=1}^2 \bar{q}_{a2}h_{1a} \bar{w} + \bar{q}_{12} \kappa f \langle Y_0,X \rangle + \sum_{a=1}^2 \bar{q}_{a1}h_{2a} \bar{w} - \bar{q}_{21} \kappa f \langle Y_0,X \rangle\\
&=0.
\end{align*}
From Hilbert's theory, $w$ can be solved for (\ref{elliptic}). By the regularity theory for elliptic equations, $w$ is smooth. We then solve $u_1$ and $u_2$ from $w$ in (\ref{umatrix}). At last, choose a point $p$ and initial value $\tau(p)$ and integrate (\ref{tau in terms of w and u}) along paths to get $\tau$.
\end{proof}

We are in the position to prove Theorem \ref{openness}.
\begin{proof}[Proof of Theorem 16]

Given a vector-valued function $\mz$, let $\my=\phi(\mz)$ be the solution of (\ref{inhomo inf def}) with the right-hand side $\bar{q}=q_{ab}(\mz).$ Note that $\my$ solves (\ref{isometric embedding equation for nearby metric}) if 
$\my$ is a fixed point of $\phi.$ We intend to apply the contraction mapping principle to find a fixed point. First of all, we need an a priori estimate of the solution of (\ref{inhomo inf def}).

\begin{lem}[\cite{Hang-Hong06}, Lemma 9.2.4]\label{key estimate}
Given $0 < \alpha < 1$, and $\mz \in T_{\mr}\mathbb{R}^3$, there exist a smooth solution $\my$ of (\ref{inhomo inf def}) and a constant $C$ depending on $\alpha$ and $\Sigma$ such that
\begin{align*}
|\my|_{2,\alpha} &\leq C \left( \left| \frac{q(\mz)}{f} \right|_{1,\alpha} + \left|\na_1(f^2 c_2) - \na_2(f^2 c_1)\right|_{\alpha}\right) \\ 
&\leq C\left( |q(\mz)|_{1,\alpha} + |\na_1 c_2 - \na_2 c_1|_{\alpha}\right).
\end{align*}
Here $c_a = \na_1 (q(\mz)_{a2}) - \na_2 (q(\mz)_{a1})$ is defined as in the proof of the previous lemma. 
\end{lem}

%%%

Observe that $\na_1 (f^2 c_2) - \na_2 (f^2 c_1)$ does not involve the third derivatives of $\mz$ and every term contains at least two $\mz$'s.  Lemma \ref{key estimate} implies
\begin{align*}
|\phi(\mz)|_{2,\alpha} \leq C_1 \left( |\sigma'_{ab} - \sigma_{ab}|_{2,\alpha} + |\mz|_{2,\alpha}^2\right).
\end{align*} 
Note that the solution is linear in $q_{ab}$. Thus, if $\my$ is the solution of $g_{ij}(\mr) \left( \mr_a^i D_b \my^j + D_a \my^i \mr_b^j \right) = q_{ab}(\mz),$ and $\my'$ is the solution of $g_{ij}(\mr) \left( \mr_a^i D_b \my'^j + D_a \my'^i \mr_b^j \right) = q_{ab}(\mw),$ then the difference $\my-\my'$ satisfies the equation with right-hand side $\bar{q}= q(\mz) - q(\mw).$ Hence, we have
\begin{align*}
|\my-\my'|_{2,\alpha} \leq C\lt(|\bar{q}|_{1,\alpha}+ |\na_1 c_2 - \na_2 c_1|_{\alpha}\rt),
\end{align*}
where $c_1,c_2$ are expressed in terms of the coefficients of $\bar{q}$. The term $ |\na_1 c_2 - \na_2 c_1|_{\alpha}$ does not involve the derivatives of $\my$ and $\my'$ of order higher than two. Given a symmetric bilinear form $A_{ij}$ and two sections $\alpha,\beta \in \Gamma(S^2,E \otimes T^*S^2),$ define $A (\alpha \cdot \beta) = \frac{1}{2} A_{ij}(\alpha_a^i \beta_b^j + \alpha_b^i \beta_a^j)$. Note that $A(d\mz \cdot d\mz - d\mw \cdot d\mw) = A(d(\mz + \mw)\cdot d( \mz -\mw))$ for any $A$. Hence we have 
\begin{align*}
q(\mz) -q(\mw) &= -g(d(\mz+\mw) \cdot d(\mz-\mw)) \\
&\quad - \Big[ (\mz-\mw)^m \left( \int_0^1 \lt( \frac{\pl F_{ijkl}}{\pl \my^m} \rt) (t\mz + (1-t)\mw) dt \right)\mz^k\mz^l \\
&\quad\qquad + F_{ijkl}(\mw) (\mz+\mw)^k (\mz-\mw)^l \Big]  \mr_a^i \mr_b^j \\
&\quad -\Big[ (\mz-\mw)^m \left( \int_0^1 \frac{\pl G_{ijk}}{\pl \my^m}(t\mz + (1-t)\mw) dt \right) \mz^k(2d\mr \cdot d\mz + d\mz \cdot d\mz) \\
&\quad\qquad + G_{ijk}(\mw) (\mz-\mw)^k (2d\mr \cdot d\mz + d\mz \cdot d\mz) \\
&\quad\qquad + G_{ijk}(\mw) \mw^k (2d\mr \cdot d(\mz-\mw) + d(\mz+\mw) \cdot d(\mz-\mw)) \Big] 
\end{align*}
If $|\mz|_{2,\alpha},|\mw|_{2,\alpha} < 1,$ then
\begin{align*}
|\phi(\mz)-\phi(\mw)|_{2,\alpha} \leq C_2 \left( |\mz|_{2,\alpha} + |\mw|_{2,\alpha } \right) |\mz-\mw|_{2,\alpha}.
\end{align*} 
If we choose $\mu <1$ such that $C_1\mu <\frac{1}{2}$ and $2C_2 \mu < 1,$ then for any metric $\sigma'$ with $C_1|\sigma' - \sigma|_{2,\alpha} < \frac{\mu}{2},$ $\phi: B_\mu \rw B_\mu$ is a contraction mapping in $C^{2,\alpha}.$ The existence of solution to (\ref{linear}) follows from contraction mapping principle. This completes the proof of Theorem \ref{openness}.
\end{proof}

From Theorem \ref{openness} and Lemma \ref{Ricci flow on surfaces}, $I$ is open and non-empty as $[T_0,\infty) \subset I$ for some large $T_0.$

\subsection{Closedness}
To prove closedness, we have to establish the a priori estimate for the isometric embedding. Suppose we have a sequence of isometric embeddings $\mr_{t_i}$ with $t_i \rw T.$ Recall that we fix a diffeomorphism $\varphi:\mathbb{H}^3_{-\kappa} \rw \mathbb{R}^3$. 
\begin{defn}
We say that a surface $\Sigma \subset \mathbb{H}^3_{-\kappa}$ {\it is centered at the origin} if $\varphi(\Sigma)$ has center of mass at (0,0,0).
\end{defn} 
\noindent By an isometry in $\mathbb{H}^3_{-\kappa},$ we may assume that the embeddings $\mr_{t_i}$ are centered at the origin. 

In two-dimensional spaces, the Ricci flow equation (\ref{normalized Ricci flow}) can be rewritten as a parabolic equation of a scalar function. By the uniformization theorem, $\bar{\sigma}= e^{2\bar{u}} \hat{\sigma}$ for a metric $\hat{\sigma}$ with constant Gauss curvature $\hat{K}$ and the same area as $\bar{\sigma}$. Let $\sigma_t = e^{2u_t}\hat{\sigma}$, then (\ref{normalized Ricci flow}) becomes an equation of $u$
\begin{align}\label{parabolic equation of the conformal factor}
\left\{ \begin{array}{rl}
\frac{\pl u}{\pl t} &= \hat{K} - K \\
u_0 &= \bar{u}.
\end{array}\right.
\end{align} 
Moreover, by the work of M. Struwe, we have the following:
\begin{thm}[\cite{Struwe02}, Theorem 6.1]\label{convergence of the parabolic equation}
For any $u_0 \in H^2(S^2,\hat{\sigma})$, there exist a unique global solution $u$ of (\ref{parabolic equation of the conformal factor}) and a smooth limit $u_\infty$ corresponding to a smooth metric $\sigma_\infty = e^{2u_\infty} \hat{\sigma}$ of constant curvature such that
\begin{align}\label{estimate of the conformal factor}
\|u(t) - u_\infty \|_{H^2} \leq C e^{-\alpha t}
\end{align}
for some constants $C$ and $\alpha$ depending only on $\hat{\sigma}$ and $u_0$
\end{thm}

From (\ref{estimate of the conformal factor}), the diameters of $\sigma_t$ are uniformly bounded. Thus, $\max_{\Sigma_{t_i}} \mr^i$ and $\max_{\Sigma_{t_i}} f$ are bounded by a constant depending only on $\bar{\sigma}$. This proves the uniform $C^0$-estimate for $\mr_{t_i}.$

The $C^1$-estimate follows from the isometric embedding equation
\begin{align*}
g_{ij} \dfrac{\pl \mr^i}{\pl u^a} \dfrac{\pl \mr^j}{\pl u^b} = \sigma_{ab}.
\end{align*}
Indeed, we may assume $g_{ij}$ is diagonal everywhere by a change of variable. Then for $i=1,2,3$, we can directly check
\begin{align*}
|\na \mr^i|^2 = \sigma^{ab}\mr^i_a\mr^i_b \leq C g_{jk} \sigma^{ab} \mr^j_a\mr^k_b = 2C
\end{align*}
where $C$ only depends on $\max_\Sigma f.$

In the following, we write $\sigma$ for $\sigma_{t_i}$ and $\mr$ for $\mr_{t_i}$ if there is no risk of confusion. The key to prove $C^2$-estimate is a uniform bound of the principal curvatures. Since $\Sigma$ is convex, it suffices to bound the mean curvature. The main difficulty lies in that the Gauss curvature of a convex surface in hyperbolic space may be negative somewhere. We remark that for the strictly convex case, the uniform bound of the mean curvature is proved by Pogorelov \cite[page 337-342]{Pogorelov73}.

\begin{thm}[Theorem B]\label{apriori estimate of mean curvature}
Let $\Sigma$ be a closed convex surface in $\mathbb{H}^3_{-\kappa}$, normalized so that $\Sigma$ is centered at the origin. Then
\begin{align*}
\max_\Sigma H \leq C,
\end{align*}
for some constant $C$ depending only on $\|f\|_{C^0(\Sigma)}$ and $\| K\|_{C^2(\Sigma)}.$
\end{thm}
\begin{proof}
Let $\lambda \geq \mu$ be the two principal curvatures. Suppose $F$ achieves its maximum at $p.$ We may assume that $\lambda > 2\mu$ at $p$; otherwise $\lambda^2 \leq 2\lambda\mu = 2(K + \kappa)$ and the estimate clearly holds. We intend to apply the maximum principle to the test function $F = \log\lambda + \alpha \frac{|X|^2}{2}$, where $\alpha$ satisfies
\begin{align*}
\alpha \min f^2 > \kappa.
\end{align*}
In the following computation, we denote the covariant derivative with respect to $\sigma$ by $;$ or $\na$. Moreover, we write $\lambda_a$ for $\lambda_{;a}$ for the gradient of principal curvatures. The first and second derivatives of $F$ are given by
\begin{align}\label{first derivative}
F_a = \frac{\lambda_a}{\lambda} + \alpha f \langle X,e_a \rangle
\end{align}
and
\begin{align}
F_{;ab} = \frac{\lambda_{;ab}}{\lambda} - \frac{\lambda_a \lambda_b}{\lambda^2} + \kappa \alpha \langle X,e_a \rangle \langle X,e_b \rangle + \alpha f^2 \sigma_{ab} - \alpha f h_{ab}\langle X,\nu \rangle.
\end{align}
We compute each term in $(H\sigma^{ij} - h^{ij})F_{;ij}$. Starting with $(H\sigma^{ab} - h^{ab}) \lambda_{;ab}.$ We have
\begin{align*}
(H\sigma^{ab} - h^{ab}) \lambda_{;ab} = \mu \lambda_{;11} + \lambda \lambda_{;22}.
\end{align*}
By (\ref{Hessian of eigenvalue}) and the Codazzi equation,
\begin{align*}
(H\sigma^{ab} - h^{ab}) \lambda_{;ab} = \mu \left( h_{11;11} + \frac{2}{\lambda - \mu}(h_{11;2})^2 \right) + \lambda \left( h_{11;22} + \frac{2}{\lambda - \mu}(h_{22;1})^2 \right).
\end{align*}
%\begin{remark}
%On the other hand,
%\begin{align*}
%(H\sigma^{ab} - h^{ab})\mu_{;ab} = \mu \left( h_{22;11} {\color{red}-} \frac{2}{\lambda-\mu}(h_{22;1})^2\right) + \lambda \left( h_{22;22} {\color{red}-} \frac{2}{\lambda-\mu}(h_{22;2})^2 \right)
%\end{align*}
%is not convex enough! That's why the test function $\frac{H}{\langle X,\nu \rangle^a}$ doesn't work.
%\end{remark}
By the Codazzi equation and commutation formula, $h_{11;22} = h_{22;11} + K(\lambda-\mu).$ Thus,
\begin{align*}
(H\sigma^{ab} - h^{ab}) \lambda_{;ab} = \mu \left( h_{11;11} + \frac{2}{\lambda - \mu}(\lambda_2)^2 \right) + \lambda \left( h_{22,11} + \frac{2}{\lambda - \mu}(\mu_1)^2 \right) + K \lambda(\lambda - \mu). 
\end{align*}
On the other hand, differentiating the Gauss equation $\det(h) = K + \kappa,$ we get
\begin{align*}
(K+\kappa)_{;ab} &= \left( (H\sigma^{cd} - h^{cd}) h_{cd;a} \right)_{;b} \notag\\
&= (H\sigma^{cd} - h^{cd})h_{cd;ab} + H_a H_b - \tensor{h}{^{cd}_{;b}}h_{cd;a}.
\end{align*}
In particular,
\begin{align}\label{Gauss equation;11}
K_{;11} &= \mu h_{11;11} + \lambda h_{22;11} + (\lambda_1 + \mu_1)^2 - (\lambda_1)^2 - 2(\lambda_2)^2 - (\mu_1)^2 \notag\\
&= \mu h_{11;11} + \lambda h_{22;11}  + 2\lambda_1\mu_1 - 2(\lambda_2)^2.
\end{align}
Therefore,
\begin{align}\label{Hessian of lambda}
(H\sigma^{ab} - h^{ab})\lambda_{;ab} = K_{;11} - 2\lambda_1\mu_1 + \frac{2\lambda}{\lambda-\mu} \left( (\lambda_2)^2 + (\mu_1)^2 \right) + K\lambda(\lambda - \mu).
\end{align}

At $p$, the derivatives $F_a =0$ and $F_{;ab} \leq 0$. We thus have
\begin{align}	
\begin{split}
\lambda_1 &= -\alpha f \langle X,e_1 \rangle \lambda,\\
\lambda_2 &= -\alpha f \langle X,e_2 \rangle \lambda,\\
\mu_1 &= \frac{K_1}{\lambda} + \frac{\alpha f (K+\kappa) \langle X,e_1 \rangle}{\lambda} = O(1/\lambda),\\
\mu_2 &= \frac{K_2}{\lambda} -  \frac{\alpha f (K+\kappa) \langle X,e_2 \rangle}{\lambda} = O(1/\lambda);
\end{split}
\end{align}
and
\begin{align*}
0 &\geq \left( H\sigma^{ab} - h^{ab}\right)F_{;ab} \\
&= \frac{1}{\lambda} \left( K_{;11} - 2\lambda_1\mu_1 + \frac{2\lambda}{\lambda-\mu} \left( (\lambda_2)^2 + (\mu_1)^2 \right) + K\lambda(\lambda - \mu) \right)\\
&\quad - \alpha^2 f^2 \langle X,e_1 \rangle^2 \mu - \alpha^2 f^2 \langle X,e_2 \rangle^2 \lambda \\
&\quad + \kappa \alpha \left( \mu\langle X,e_1\rangle^2 + \lambda \langle X,e_2 \rangle^2 \right) + \alpha f^2 (\lambda + \mu) - 2 \alpha f (K+\kappa) \langle X,\nu \rangle \\
&= \left( 2 \alpha^2 f^2 \langle X,e_2 \rangle^2 +K - \alpha^2 f^2 \langle X,e_2 \rangle^2 + \kappa\alpha \langle X,e_2 \rangle^2 + \alpha f^2 \right) \lambda+ O(1) \\
&\geq (K + \alpha f^2) \lambda + O(1).
\end{align*}
Note that the first term in the last equality comes from $(\lambda_2)^2$. Here we say a function $G = O(\lambda^p)$ if there exist some constants $c$ and $C$ depending only on $\|K\|_{C^2(\Sigma)}$ and  $\|f\|_{C^0(\Sigma)}$ such that $c\lambda^p \leq G \leq C\lambda^p$ when $\lambda \geq 1.$

From our assumption on $\alpha,$ $\lambda(p) \leq C.$ For other points $q \in \Sigma,$
\begin{align*}
\lambda(q) \leq \lambda(p) \frac{e^{|X|^2}(p)}{e^{|X|^2}(q)} \leq C.
\end{align*}
\end{proof}

We are in the position to prove the $C^2$-estimate. Writing $D_{\mr_b} \mr_a$ in two ways
\begin{align*}
\left( D_{\mr_b} \mr_a \right)^i &= \dfrac{\pl^2 \mr^i}{\pl u^a \pl u^b} + \Gamma^i_{jk} \dfrac{\pl \mr^j}{\pl u^a} \dfrac{\pl \mr^k}{\pl u^b}\\
&= \Gamma^c_{ab} \dfrac{\pl \mr^i}{\pl u^c} - h_{ab} \nu^i,
\end{align*}
we obtain
\begin{align}
\na_b \na_a \mr^i = -h_{ab}\nu^i - \Gamma^i_{jk} \dfrac{\pl \mr^j}{\pl u^a} \dfrac{\pl \mr^k}{\pl u^b} \label{hessian of position vector}
\end{align} 
Hence $\|\mr^i\|_{C^2} \leq C$ where $C$ depends on the upper bound of principal curvatures and $\|\mr^i\|_{C^1}.$ By Arzela-Ascoli theorem, a subsequence of $\mr_{t_i}$ converges to some $\mr_T \in C^{1,1}.$ This completes the proof of Theorem \ref{isom}.

When $K > \kappa,$ the continuity method actually produces a smooth isometric embedding.
\begin{thm}\label{isom>0}
Let $\bar{\sigma}$ be a smooth metric on $S^2$ with Gauss curvature $K >-\kappa$. Then there exists a smooth isometric embedding $i: (S^2, \bar{\sigma}) \rw \mathbb{H}^3_{-\kappa}$ which is unique up to congruence. 
\end{thm}
\begin{proof}
The proof of uniqueness (indepent of Pororelov's) could be found in \cite{Guan-Shen13}. To prove the theorem, we have to establish a priori estimates for the higher derivatives of $\mr_{t_i}$. Let $\Sigma_{t_i}$ denote $\mr_{t_i}(\Sigma).$ Define 
$$
\rho(t_i) = \frac{1}{2} \langle X|_{\Sigma_{t_i}},X|_{\Sigma_{t_i}} \rangle.
$$
We compute
\begin{align}
\rho_{;a} &= f \langle X, \frac{\pl \mr}{\pl u^a} \rangle, \notag\\
\rho_{;ab} &= \kappa \langle X,\frac{\pl \mr}{\pl u^a} \rangle \langle X,\frac{\pl \mr}{\pl u^b} \rangle + f\sigma_{ab} -h_{ab}\langle X,\nu\rangle. \label{hessian of rho}
\end{align} 
Taking the determinant of (\ref{hessian of rho}), we get
\begin{align}
\mathcal{F} \equiv \det(\rho_{;ab} - f\sigma_{ab}) - (K + \kappa)(2\rho - f^2|\na\rho|^2) =0 \label{fully nonlinear elliptic equaion of rho}
\end{align}
The assumption $K > -\kappa$, together with Lemma \ref{lower bound of inner radius} and Lemma \ref{lower bound of X dot nu}, imply that (\ref{fully nonlinear elliptic equaion of rho}) is uniformly elliptic
\begin{align*}
\mathcal{F}_{\rho_{;11}} \mathcal{F}_{\rho_{;22}} - \mathcal{F}_{\rho_{;12}}^2 = \frac{K+\kappa}{\det \sigma} \langle X,\nu \rangle^2 >0.
\end{align*}
By \cite[Theorem I] {Nirenberg53: Holder continuity} (see also \cite[Lemma 9.3.4]{Hang-Hong06}) and the Schauder estimates, $\|\rho\|_{C^{m,\alpha}}$ is uniformly bounded for any $m$ and $0 < \alpha < 1$. The higher regularity of $\mr$ follows from (\ref{hessian of position vector}) and (\ref{hessian of rho}).

By (\ref{hessian of rho}), since the support function $\langle X,\nu \rangle$ is bounded from above and below, the $m$-th derivatives of second fundamental form $\na^m h$ are bounded by the $(m+2)$-th derivatives of $\rho$ and $m+1$-th derivatives of $\mr.$ Furthermore, by (\ref{hessian of position vector}), the $(m+2)$-derivatives of $\mr$ is bounded by the $m$-derivatives of $h$ and the $(m+1)$-th derivatives of $\mr$. The higher regularity of $\mr$ follows. This completes the proof.
\end{proof}

\begin{lem}\label{lower bound of inner radius}
Suppose that $\Sigma$ is a smooth closed convex surface centered at the origin in $\mathbb{H}^3_{-\kappa}$. Then there exists a positive constant $R$ depending on $\frac{1}{\max K+\kappa}$ and the diameter of $\Sigma$ such that the geodesic ball of radius $R$ at the origin lies inside $\Sigma.$
\end{lem}
\begin{proof}
We identify the hyperbolic space $\mathbb{H}^3_{-\kappa}$ with the hyperboloid $\{(x^0,x^1,x^2,x^3) : -(x^0)^2 + (x^1)^2 + (x^2)^2 + (x^3)^2 = -\frac{1}{\kappa}\}$. Consider the Beltrami map
\begin{align*}
\begin{array}{rll}
\beta: \mathbb{H}^3_{-\kappa} &\rw& \{x^0 = \frac{1}{\sqrt{\kappa}} \} \\
(x^0,x^1,x^2,x^3)&\mapsto& \frac{1}{\sqrt{\kappa}}\left( 1,\frac{x^1}{x^0},\frac{x^2}{x^0},\frac{x^3}{x^0} \right).
\end{array}
\end{align*}
We identify $\{x^0 = \frac{1}{\sqrt{\kappa}} \}$ with the Euclidean space $\mathbb{R}^3.$ Recall the static potential $f$ is equal to $\sqrt{\kappa}x^0.$ Suppose $\Sigma$ is given by the embedding
\begin{align*}
\mathbf{r}(u^1,u^2) = \left( x^0(u^1,u^2),x^1(u^1,u^2),x^2(u^1,u^2),x^3(u^1,u^2)\right)
\end{align*}
with metric 
\begin{align*}
\sigma_{ab} = - \frac{1}{\kappa} \frac{\pl f}{\pl u^a}\frac{\pl f}{\pl u^b} + \sum_{i=1}^3 \frac{\pl x^i}{\pl u^a} \frac{\pl x^i}{\pl u^b},
\end{align*}
then the embedding of $\tilde{\Sigma} = \beta(\Sigma)$ is given by $\frac{\mathbf{r}}{f} - \frac{\pl}{\pl x^0}$ where $\frac{\pl}{\pl x^0} = (1,0,0,0)$. We compute the induced metric of $\tilde{\Sigma}$  
\begin{align*}
\tilde{\sigma}_{ab} &= \left\langle \frac{1}{f} \frac{\pl \mathbf{r}}{\pl u^a} - \frac{1}{f^2}\frac{\pl f}{\pl u^a}\mathbf{r}, \frac{1}{f} \frac{\pl \mathbf{r}}{\pl u^b} - \frac{1}{f^2}\frac{\pl f}{\pl u^b}\mathbf{r} \right\rangle \\
&= \frac{1}{f^2} \left( \sigma_{ab} - \dfrac{\frac{\pl f}{\pl u^a}\frac{\pl f}{\pl u^b}}{\kappa f^2} \right)
\end{align*} 
It is not hard to check that the unit normals of $\Sigma$ and $\tilde{\Sigma}$ are related by
\begin{align*}
\tilde{\nu} = \dfrac{\nu + \langle \nu, \frac{\pl}{\pl x^0} \rangle \frac{\pl}{\pl x^0}}{\sqrt{1 + \langle \nu, \frac{\pl}{\pl x^0}\rangle^2}}
\end{align*}
Next we compute the second fundamental forms of $\Sigma$ and $\tilde{\Sigma}$ \begin{align*}
h_{ab} &= -\left\langle \dfrac{\pl^2 \mathbf{r}}{\pl u^a\pl u^b}, \nu \right\rangle\\
\tilde{h}_{ab} &= -\left\langle \frac{\pl}{\pl u^b} \left( \frac{1}{f} \frac{\pl \mathbf{r}}{\pl u^a} - \frac{1}{f^2}\frac{\pl f}{\pl u^a}\mathbf{r} \right) ,\frac{\nu}{\sqrt{1+\langle \nu,\frac{\pl}{\pl x^0}\rangle^2}} \right\rangle \\
&= \frac{1}{\sqrt{1 + \langle \nu, \frac{\pl}{\pl x^0} \rangle^2}} \frac{h_{ab}}{f}.
\end{align*}
Hence $\tilde{\Sigma}$ is convex. The Gauss curvatures of $\Sigma$ and $\tilde{\Sigma}$ are related by  
\begin{align*}
\tilde{K} &= \dfrac{\det(\tilde{h}_{ab})}{\det(\tilde{\sigma}_{ab})} \\
&= \dfrac{f^2}{(1+\langle \nu,\frac{\pl}{\pl x^0} \rangle^2)(1- \frac{|\na f|^2}{\kappa f^2})}(K + \kappa)
\end{align*}
From (\ref{grad f}), we have
\begin{align*}
1 - \frac{|\na f|^2}{\kappa f^2} = \frac{1 + \kappa\langle X,\nu \rangle^2}{f^2} \end{align*}
Hence $\tilde{K} \leq (\max K + \kappa)(\max f)^4$. We apply \cite[Lemma 9.1.1]{Hang-Hong06} to conclude that there exists a positive constant $R$ depending only on $\frac{1}{\max \tilde{K}}$ and the diameter of $\tilde{\Sigma}$ such that there exists a ball of radius $R$ inside $\tilde{\Sigma}.$ Therefore, there exists a positive constant $R'$ depending only on $\frac{1}{\max K + \kappa}$ and the diameter of $\Sigma$ such that there exists a ball of radius $R'$ inside $\Sigma.$
\end{proof}

\begin{lem}\label{lower bound of X dot nu}
For a convex surface $\Sigma$ in $\mathbb{H}^3_{-\kappa},$ $\min_\Sigma \langle X,\nu \rangle = \min_\Sigma r.$
\end{lem}
\begin{proof}
At the critical points of $\langle X,\nu\rangle,$
\begin{align*}
0 = \na_a \langle X,\nu \rangle = h(X^T, \cdot).
\end{align*}
Since $\Sigma$ is convex, we have $X^T=0.$ 
\end{proof}

%%%%%%%%%%%%%
\section{ Isometric embeddings into the anti-de Sitter spacetime}
%The Friedmann-Robertson-Walker metric}
We consider the isometric embedding problem of $(S^2,\sigma)$ into the anti-de Sitter spacetime $(AdS, g_{AdS})$. We work on the cosmological chart of $AdS$ on which the metric $g_{AdS}$ can be expressed as
\[
g_{AdS}=-dt^{2}+S^2(t) g, 
\]
where $S(t) = \cos( \sqrt{\kappa}t)$ and $g$ is the hyperbolic metric with sectional curvature $-\kappa$ \cite[(5.9)]{Hawking-Ellis73}.

We have the following theorem:
\begin{thm}[Theorem C]
Given a smooth metric $\sigma$ and a smooth function $s$ on $S^2.$ Suppose 
\begin{align}
K +  \frac{S'}{S} \Delta s - \lt( \frac{S''}{S} - \frac{S'^2}	{S^2} \rt) |\nabla s|^2  + 
 ( 1+|\nabla s|^2 )^{-1} \lt(\frac{ \det(\nabla^2 s) }{\det \sigma} - \frac{S'}{S} \nabla^a s \nabla^b s\nabla_a \nabla_b s\rt)>-\kappa. \label{Gauss curvature of the projection}
\end{align}
where $\na$ and $\Delta$ denote the gradient and Laplace operator with respect to $\sigma.$ Then there exists a unique spacelike isometric embedding $\mathbf{r}:(S^2,\sigma) \rw AdS$ with prescribed cosmological time function $s$.
\end{thm}
\begin{proof}
Suppose we have an isometric embedding into the cosmological chart $\mr = (s, \mr^1, \mr^2, \mr^3): (S^2,\sigma) \rw AdS$. Consider the projection $\hat{\mr}: S^2 \rw \mathbb{H}^3_{-\kappa}$, where $\hat{\mr} = (\mr^1, \mr^2, \mr^3)$. The induced metric on the image of the embedding satisfies
$$
\hat{\sigma}=S^{-2}\lt(s)(ds^{2}+\sigma\rt)=S^{-2}(s)\bar{\sigma},
$$
where $\bar{\sigma} = ds^{2}+\sigma$.

Denote the Gauss curvatures of $\sigma,\bar{\sigma}$, and $\hat{\sigma}$ by $K,\bar{K},$ and $\hat{K}$, respectively. The Gauss curvatures $\hat{K}$ and $\bar{K}$ can be computed as
\begin{eqnarray*}
\hat{K} &=& S^2(s) \lt( \bar{K} + \Delta_{\bar{\sigma}} \ln S(s) \rt),\\
\bar{K} &=& ( 1 + |\nabla s |^2 )^{-1} \lt( K + ( 1 + |\nabla s |^2 )^{-1} \frac{\det (\nabla^2 s)}{\det \sigma} \rt),
\end{eqnarray*}
where $\Delta_{\bar{\sigma}}$ is the Laplacian with respect to $\bar{\sigma}$  and $\na^2$ denotes the Hessian with respect to $\sigma.$ Moreover, 
\begin{eqnarray*}
\Delta_{\bar{\sigma}} \ln S(s) 
	&=& \lt( \sigma^{ab} - \frac{\nabla^a s \nabla^b s}{ 1+ |\nabla s|^2} \rt) 
			\lt( \frac{S'}{S} \cdot \frac{ \nabla_a \nabla_b s}{ 1+|\nabla s|^2} - \lt( \frac{S''}{S} - \frac{S'^2}{S^2} \rt) \nabla_a s \nabla_b s \rt) \\
	&=& \frac{S'}{S} \cdot \frac{\Delta s}{ 1+|\nabla s|^2} - \lt( \frac{S''}{S} - \frac{S'^2}{S^2} \rt) \frac{ |\nabla s|^2}{1+|\nabla s|^2} - \frac{S'}{S} \cdot \frac{ \nabla^a s \nabla^b s \nabla_a \nabla_b s}{ (1+|\nabla|^2 )^2}.
\end{eqnarray*}
Thus,
\begin{eqnarray*}
\hat{K} &=& \frac{ S^2(s) }{ ( 1 + |\nabla s |^2 )} \bigg\{ K +  \frac{S'}{S} \Delta s - \lt( \frac{S''}{S} - \frac{S'^2}	{S^2} \rt) |\nabla s|^2  + \\
	&&\qquad ( 1+|\nabla s|^2 )^{-1} \lt(\frac{ \det(\nabla^2 s) }{\det \sigma} - \frac{S'}{S} \nabla^a s \nabla^b s\nabla_a \nabla_b s\rt)\bigg\}.
\end{eqnarray*} 

The above computation shows that if (\ref{Gauss curvature of the projection}) holds, then by Theorem \ref{isom>0}, there exists an isometric embedding $(\mr^1, \mr^2, \mr^3): (S^2, \hat{\sigma}) \rw \mathbb{H}^3_{-\kappa} $ and $(s, \mr^1, \mr^2, \mr^3): (S^2, \sigma) \rw AdS$ is the desired isometric embedding into anti-de Sitter spacetime. 
 
As for the uniqueness, we assume that $\mr_a  = (s, \mr^1_a,\mr^2_a,\mr^3_a): (S^2, \sigma) \rw AdS, a=1,2$ are two isometric embeddings. Since the induced metrics of the projection $\hat{\mr}_1$ and $\hat{\mr}_2$ are isometric, they differ by an isometry in $\mathbb{H}^3_{-\kappa}$. Consequently, $\mr_1$ and $\mr_2$ are congruent.
\end{proof}
\appendix
\appendixpage
\section{Derivatives of eigenvalues}
The purpose of this appendix is to prove a special case of the following well-known fact(\cite[Theorem 5.5]{Ball84})
\begin{prop}
Let $M: \mathbb{R}^m \rw \{ \mbox{symmetric } n\times n \mbox{ matrices} \}$ be a smooth matrix-valued function with distinct eigenvalues $\lambda_1(x), \ldots, \lambda_n(x)$. Suppose $M(0)$ is diagonal. Then we have
\begin{align}
\frac{\pl \lambda_i}{\pl x^a}(0)&= \frac{\pl M_{ii}}{\pl x^a} (0) \label{derivative of eigenvalue}\\
\frac{\pl^2 \lambda_i}{\pl x^a \pl x^b}(0) &= \frac{\pl^2 M_{ii}}{\pl x^a \pl x^b} (0) - 2 \sum_{j \neq i} \dfrac{\frac{\pl M_{ij}}{\pl x^a}(0) \frac{\pl M_{ij}}{\pl x^b} (0)}{\lambda_j(0) - \lambda_i(0)} \label{Hessian of eigenvalue}
\end{align}
\end{prop}
\begin{proof}
Denote the adjoint matrix of $M$ by $M^*.$ By definition,
\begin{align}\label{characteristic equation}
0 = \det (M(x) - \lambda_i(x)I)
\end{align}
Differentiating (\ref{characteristic equation}), we get
\begin{align*}
0 = \mbox{Tr}\left( (M - \lambda_i I)^* \left( \frac{\pl M}{\pl x^a} - \frac{\pl \lambda_i}{\pl x^a}I \right)\right).
\end{align*}
Since $M(0)$ is diagonal, the only nonzero entry of $( M(0) - \lambda_i(0) I)^*$ is 
\begin{align}\label{adjoint matrix at 0}
(M(0) - \lambda_i(0) I)^*_{ii} = \prod_{k \neq i}(\lambda_k - \lambda_i)(0).
\end{align}
From (\ref{adjoint matrix at 0}), the first statement follows. To prove the second statement, we differentiate (\ref{characteristic equation}) twice to get 
\begin{align*}
0 = \mbox{Tr} \left( \frac{\pl (M - \lambda_i I)^*}{\pl x^b} \left( \frac{\pl M}{\pl x^a} - \frac{\pl \lambda_i}{\pl x^a} I \right) + (M -\lambda_i I)^*\left( \frac{\pl^2 M}{\pl x^a \pl x^b} - \frac{\pl^2 \lambda_i}{\pl x^a \pl x^b} I \right)\right).
\end{align*}
By differentiating the equation $(M - \lambda_i I)^*(M - \lambda_i I) =0,$ we observe that the only nonzero entries of $\frac{\pl (M-\lambda_i I)^*}{\pl x^b}(0)$ are
\begin{align}\label{derivative of adjoint matrix at 0}
\left[ \frac{\pl (M-\lambda_i I)^*}{\pl x^b}(0) \right]_{ij} = (-1) \left( \prod_{k \neq i}(\lambda_k(0) - \lambda_i(0)) \right) \dfrac{\frac{\pl M_{ij}}{\pl x^b}(0)}{\lambda_j(0) - \lambda_i(0)} \quad \mbox{for } j \neq i.
\end{align}
From (\ref{derivative of adjoint matrix at 0}), we obtain
\begin{align*}
0 = \left( \prod_{k \neq i}(\lambda_k(0) - \lambda_i(0)) \right) \left( 2 \sum_{j \neq i} (-1) \dfrac{\frac{\pl M_{ij}}{\pl x^b}(0)}{\lambda_j(0) -\lambda_i(0)} \dfrac{\pl M_{ij}}{\pl x^a}(0) + \frac{\pl^2 M_{ii}}{\pl x^b \pl x^a}(0) - \frac{\pl^2 \lambda_i}{\pl x^b \pl x^a}(0) \right)
\end{align*}
This proves the second statement.
\end{proof}

\end{document}